\documentclass[12pt,twoside]{amsart}
\usepackage{amssymb}
\usepackage{amscd}
\usepackage{xypic}

\title{On the ACC for lengths of extremal rays}
\author{Osamu Fujino}
\author{Yasuhiro Ishitsuka} 
\subjclass[2010]{Primary 14M25; Secondary 14E30.}
\date{2012/6/4, version 1.21}
\keywords{ascending chain condition, lengths of 
extremal rays, Fano varieties, toric varieties, 
minimal model program}
\address{Department of Mathematics, Faculty of Science, 
Kyoto University, Kyoto 606-8502, Japan}
\email{fujino@math.kyoto-u.ac.jp}
\address{Department of Mathematics, Faculty of Science, 
Kyoto University, Kyoto 606-8502, Japan}
\email{yasu-ishi@math.kyoto-u.ac.jp}
\newcommand{\lcm}[0]{{\operatorname{lcm}}}
\newcommand{\Star}[0]{{\operatorname{Star}}}
\newcommand{\mult}[0]{{\operatorname{mult}}}
\newtheorem{thm}{Theorem}[section]
\newtheorem{lem}[thm]{Lemma}

\newtheorem{conj}[thm]{Conjecture}
\newtheorem{prop}[thm]{Proposition}
\newtheorem*{claim}{Claim}

\theoremstyle{definition}
\newtheorem{ex}[thm]{Example}
\newtheorem{defn}[thm]{Definition}

\newtheorem*{ack}{Acknowledgments}       
\newtheorem{say}[thm]{}
\begin{document}
\bibliographystyle{amsalpha+}

\begin{abstract} 
We discuss the ascending chain condition for 
lengths of extremal rays. 
We prove that the lengths of extremal rays 
of $n$-dimensional 
$\mathbb Q$-factorial toric Fano varieties with Picard number one 
satisfy the ascending chain condition. 
\end{abstract}

\maketitle

\section{Introduction}

We discuss the ascending chain condition (ACC, for short) 
for (minimal) lengths of extremal rays. 

First, let us recall the definition of {\em{$\mathbb Q$-factorial 
log canonical Fano varieties with Picard number one}}. 

\begin{defn}[$\mathbb Q$-factorial log canonical Fano varieties with 
Picard number one]
Let $X$ be a normal projective 
variety with only log canonical singularities. 
Assume that $X$ is $\mathbb Q$-factorial, $-K_X$ is ample, 
and $\rho (X)=1$. 
In this case, we call $X$ a {\em{$\mathbb Q$-factorial log canonical 
Fano variety with Picard number one}}. 
\end{defn}

\begin{defn}[(Minimal) lengths of extremal rays]
Let $(X, \Delta)$ be a log canonical pair and let $f:X\to Y$ be a projective 
surjective morphism. 
Let $R$ be a $(K_X+\Delta)$-negative extremal ray 
of $\overline {NE}(X/Y)$. 
Then 
$$
\min_{[C]\in R}\left(-(K_X+\Delta)\cdot C\right)
$$
is called the {\em{$($minimal$)$ length of the $(K_X+\Delta)$-negative 
extremal ray $R$}}. 
\end{defn}

From now on, we want to discuss the following conjecture. 
It seems to be the first time that the ascending chain condition 
for lengths of extremal rays is discussed in the literature. 

\begin{conj}[ACC for lengths of extremal rays of 
$\mathbb Q$-factorial log canonical Fano varieties with Picard number one]
\label{conj12}
We set 
$$
\mathcal L_n:=\left\{ l(X) \ ; \ 
\begin{matrix}
{\text{$X$ is 
an $n$-dimensional $\mathbb Q$-factorial log canonical}}
\\ {\text{Fano variety with Picard number one}}.
\end{matrix} 
\right\}
$$ 
such that 
$$
l(X):=\min _{C}(-K_X\cdot C)
$$
where $C$ is an integral curve on $X$. 
For every $n$, the set $\mathcal L_n$ satisfies the ascending chain condition. 
This means that if $X_k$ is an $n$-dimensional 
$\mathbb Q$-factorial log canonical Fano variety with Picard number one 
for every $k$ such that 
$$
l(X_1)\leq l(X_2)\leq \cdots \leq l(X_k)\leq \cdots
$$ 
then there is a positive integer $l$ such that 
$l(X_m)=l(X_l)$ for every $m\geq l$. 
\end{conj}

We note that $l(X)\leq 2\dim X$ when 
$X$ is a $\mathbb Q$-factorial log canonical 
Fano variety with $\rho(X)=1$ (see, for example, \cite[Theorem 
18.2]{fujino-funda}). 

Although, for inductive treatments, it may be better to 
consider the ascending chain condition for lengths of extremal rays 
of {\em{log Fano}} pairs $(X, D)$ such that 
the coefficients of $D$ are 
contained in a set satisfying the descending chain condition, 
we only discuss the case when $D=0$ for simplicity.  
In this paper, we are mainly interested in 
$\mathbb Q$-factorial toric Fano varieties with Picard number one. 
Note that a $\mathbb Q$-factorial toric variety always has only log canonical 
singularities. 
So, we define 
$$
\mathcal L^{\mathrm{toric}}_n:=\left\{ l(X) \ ;\ 
\begin{matrix} {\text{$X$ is 
an $n$-dimensional $\mathbb Q$-factorial toric}} 
\\
{\text{Fano variety with 
Picard number one}}
\end{matrix}
\right\}. 
$$

Let $X$ be an $n$-dimensional $\mathbb Q$-factorial 
toric Fano variety with $\rho (X)=1$. 
Then we have $l(X)\leq n+1$. 
Furthermore, $l(X)\leq n$ if 
$X\not \simeq \mathbb P^n$ 
(cf.~\cite[Proposition 2.9]{fujino}). 
We can easily see that 
$X\simeq \mathbb P(1, 1, 2, \cdots, 2)$ if and 
only if $l(X)=n$ 
(cf.~\cite[Section 2]{fujino}, \cite[Proposition 2.1]{fujino-osaka}, 
and \cite{fujino2}). 

The following result is the main theorem of this paper, 
which supports Conjecture \ref{conj12}. 

\begin{thm}[Main theorem]\label{main-thm}
For every $n$, 
$\mathcal L_n^{\mathrm{toric}}$ satisfies the ascending chain condition. 
\end{thm}

In 2003, Professor Vyacheslav Shokurov explained his ideas on 
minimal log discrepancies, log canonical thresholds, and 
lengths of extremal rays to the first author at his office. 
He pointed out some analogies among them and 
asked the ascending chain condition for lengths of extremal rays. 
It is a starting point of this paper. For his 
ideas on minimal log discrepancies and log canonical 
thresholds, see, for example, \cite{bs}. 
We note that Hacon--M\textsuperscript{c}Kernan--Xu announced that 
they have established the ACC for log canonical 
thresholds (cf.~\cite{hmx}). 
We also note that the ACC for minimal log discrepancies is closely related to the termination 
of log flips (cf.~\cite{shokurov}).  
We recommend the reader to see \cite{kollar} and \cite{totaro} 
for various aspects of log canonical thresholds. 

We close this section with examples. 
Example \ref{abab} shows that the set $\mathcal L_n^{\mathrm {toric}}$ does not 
satisfy the descending chain condition. 
Example \ref{abcd} implies that the ascending chain condition does not necessarily 
hold for (minimal) lengths of extremal rays of birational type. 

\begin{ex}\label{abab}
We consider $X_k=\mathbb P(1, k-1, k)$ with $k\geq 2$. 
Then 
$$
l(X_k)=\frac{2}{k-1}. 
$$
Therefore, $l(X_k)\to 0$ when $k\to \infty$. 
\end{ex}

\begin{ex}\label{abcd}
We fix $N=\mathbb Z^2$ and let $\{e_1, e_2\}$ be the
standard basis of $N$. We consider the cone $\sigma =\langle e_1, 
e_2\rangle$ in $N'=N+\mathbb Z e_{3}$, where $e_{3}=
\frac{1}{b}(1, a)$.
Here, $a$ and $b$ are positive integers such that $\gcd(a,b)=1$.
Let $Y=X(\sigma)$ be the associated affine toric surface which has only one
singular point $P$.
We take a weighted blow-up of $Y$ at $P$ with
the weight $\frac{1}{b}(1, a)$.
This means that we divide $\sigma$ by $e_{3}$ and obtain 
a fan $\Delta$ of $N'_{\mathbb R}$.
We define $X=X(\Delta)$. It is obvious that $X$
is $\mathbb Q$-factorial and $\rho (X/Y)=1$. We can easily
obtain $$K_X=f^*K_Y+\left(\frac{1+a}{b}-1\right)E, $$ 
where $E=V(e_{3})\simeq
\mathbb P^{1}$ is the exceptional curve of $f$, 
and $$-K_X\cdot E=1-\frac{b-1}{a}.$$ We note that 
$$
-K_X\cdot E=\min _{C} (-K_X\cdot C)
$$ 
where $C$ is a curve on $X$ such that $f(C)$ is a point. 
We also note that $\overline {NE}(X/Y)=NE(X/Y)$ is spanned by 
$E$. 
In the above construction, we set $a=k^2$ and $b=mk+1$ for any positive integers 
$k$, $m$. 
Then it is obvious that $\gcd(a,b)=1$. 
Thus we obtain 
$$-K_X\cdot E=1-\frac{m}{k}. $$
Therefore, the minimal lengths of $K_X$-negative extremal rays do not 
satisfy the ascending chain condition in this local setting. 
More precisely, 
the minimal lengths of $K_X$-negative extremal rays can take any values in 
$\mathbb Q\cap (0, 1)$ in this example. 
\end{ex}

We note that the minimal length of the $K_X$-negative 
extremal ray associated to a toric {\em{birational}} contraction morphism 
$f:X\to Y$ is bounded by $\dim X-1$ (cf.~\cite{fujino2}). 

For estimates of lengths of extremal rays of toric varieties and related topics, 
see \cite{fujino}, \cite{fujino-osaka}, and \cite{fujino2}. 

\begin{ack}
The first author was partially supported by The Inamori Foundation and by the 
Grant-in-Aid for Young Scientists (A) $\sharp$20684001 from JSPS. 
He would like to thank Professor Vyacheslav Shokurov for explaining his ideas 
at Baltimore in 2003. 
The both authors would like to thank Professor Tetsushi Ito 
for warm encouragement. 
They also would like to thank the referee for useful comments and 
pointing out some ambiguities. 
\end{ack}

\section{Preliminaries}\label{sec2}

In this section, we prepare various definitions and notation. 
We recommend the reader to see \cite[Section 2]{fujino} 
for basic calculations.
 
\begin{say}
Let $N\simeq \mathbb Z^n$ be a lattice of rank $n$. 
A toric variety $X(\Delta)$ is associated to a {\em{fan}} $\Delta$, 
a collection of convex cones $\sigma\subset N_\mathbb R =
N\otimes _{\mathbb Z}\mathbb R$ satisfying the following conditions: 
\begin{enumerate}
\item[(i)] 
Each convex cone $\sigma$ is a rational polyhedral in the sense 
there are finitely many 
$v_1, \cdots, v_s\in N\subset N_{\mathbb R}$ such 
that 
$$
\sigma=\{r_1v_1+\cdots +r_sv_s; \ r_i\geq 0\}=
:\langle v_1, \cdots, v_s\rangle, 
$$ 
and 
it is strongly convex in the sense 
$$
\sigma \cap -\sigma=\{0\}. 
$$
\item[(ii)] Each face $\tau$ of a convex cone $\sigma\in \Delta$ 
is again an element in $\Delta$. 
\item[(iii)] The intersection of two cones in $\Delta$ is a face of 
each. 
\end{enumerate}

\begin{defn}\label{saisyo}
The {\em{dimension}} $\dim \sigma$ of $\sigma$ is 
the dimension of the linear space 
$\mathbb R\cdot \sigma=\sigma +(-\sigma)$ spanned 
by $\sigma$. 

We denote by $N_\sigma$ the sublattice of $N$ generated 
(as a subgroup) by $\sigma\cap N$, i.e.,  
$$
N_{\sigma}:=\sigma\cap N+(-\sigma\cap N). 
$$

If $\sigma$ is a $k$-dimensional simplicial 
cone, and $v_1,\cdots, v_k$ are the 
first lattice points along the edges of $\sigma$, 
the {\em{multiplicity}} of $\sigma$ is defined 
to be the {\em{index}} of the lattice 
generated by the $\{v_i\}$ in the lattice $N_{\sigma}$; 
$$
\mult (\sigma):=|N_{\sigma}:\mathbb Zv_1+\cdots +
\mathbb Zv_k|. 
$$ 
We note that $X(\sigma)$, which is an affine toric variety 
associated to $\sigma$, is non-singular if and only 
if $\mult (\sigma)=1$. 
\end{defn}
\end{say}

Let us recall a well-known fact. See, for example, 
\cite[Lemma 14-1-1]{ma}. 

\begin{lem}\label{kantan}
A toric variety $X(\Delta)$ is $\mathbb Q$-factorial 
if and only if each cone $\sigma\in \Delta$ is simplicial. 
\end{lem}

\begin{say}
The {\em{star}} of a cone $\tau$ can be defined abstractly 
as the set of cones $\sigma$ in $\Delta$ that 
contain $\tau$ as a face. Such cones $\sigma$ are 
determined by their images in 
$N(\tau):=N/{N_{\tau}}$, that is, by 
$$
\overline \sigma=(\sigma+(N_{\tau})_{\mathbb R})/ 
(N_{\tau})_{\mathbb R}\subset N(\tau)_{\mathbb R}. 
$$ 
These cones $\{\overline \sigma ; \tau\prec \sigma\}$ 
form a fan of $N(\tau)$, and we denote this fan by 
$\Star(\tau)$. 
We set $V(\tau)=X(\Star (\tau))$. 
It is well known that $V(\tau)$ is an $(n-k)$-dimensional 
closed torus invariant subvariety of $X(\Delta)$, where $\dim \tau=k$. 
If $\dim V(\tau)=1$ (resp.~$n-1$), then we call $V(\tau)$ 
a {\em{torus invariant curve}} (resp.~{\em{torus invariant 
divisor}}). 
For the details about the correspondence between $\tau$ and 
$V(\tau)$, see \cite[3.1 Orbits]{fulton}. 
\end{say}

\begin{say}[Intersection Theory]\label{int}
Assume that $\Delta$ is simplicial. 
If $\sigma, \tau\in \Delta$ span $\gamma$ with 
$\dim \gamma=\dim \sigma +\dim \tau$, 
then 
$$
V(\sigma)\cdot V(\tau)=\frac{\mult (\sigma)\cdot \mult(\tau)}
{\mult (\gamma)} V(\gamma)
$$ 
in the {\em{Chow group}} $A^{*}(X)_{\mathbb Q}$. 
For the details, see \cite[5.1 Chow groups]{fulton}. 
If $\sigma$ and $\tau$ are contained in no cone of 
$\Delta$, then $V(\sigma)\cdot V(\tau)=0$. 
\end{say}

\begin{say}[$\mathbb Q$-factorial toric Fano varieties with Picard number one]
\label{tuitui}
Now we fix $N\simeq \mathbb Z ^n$. Let $\{v_1,\cdots,v_{n+1}\}$ 
be a set of primitive vectors such that $N_{\mathbb R}=\sum _i 
\mathbb R_{\geq 0}v_i$. 
We define $n$-dimensional cones 
$$
\sigma_i:=\langle v_1,\cdots,v_{i-1},v_{i+1},\cdots,v_{n+1}\rangle 
$$ 
for $1\leq i\leq n+1$. 
Let $\Delta$ be the complete fan generated by $n$-dimensional 
cones $\sigma_i$ and their faces for all $i$. Then 
we obtain a complete toric variety $X=X(\Delta)$ with 
Picard number $\rho (X)=1$. 
It is well known that $X$ has only log canonical singularities 
(see, for example, \cite[Proposition 14-3-2]{ma}) 
and that $-K_X$ is ample. 
We call it a {\em{$\mathbb Q$-factorial 
toric Fano variety with Picard number one}} (see 
also Lemma \ref{2727} below). 
We define $(n-1)$-dimensional cones $\mu_{i,j}=\sigma _i\cap \sigma _j$ 
for $i\ne j$. 
We can write $\sum _i a_i v_i=0$, where $a_i\in \mathbb Z_{>0}$ for every $i$ and 
$\gcd(a_1,\cdots,a_{n+1})=1$. 
From now on, we simply write $V(v_i)$ to denote 
$V(\langle v_i\rangle)$ for every $i$. Note that 
$\mult (\langle v_i\rangle)=1$ for every $i$. 
Then we obtain 
$$
0< V({v_{l}})\cdot V(\mu_{k,l})=\frac{\mult {(\mu_{k,l})}}
{\mult {(\sigma_{k})}},  
$$
$$
V({v_{i}})\cdot V(\mu_{k,l})=\frac{a_i}{a_{l}}\cdot
\frac{\mult {(\mu_{k,l})}}
{\mult {(\sigma_{k})}}, 
$$
and 
\begin{eqnarray*}
-K_{X} \cdot V(\mu_{k,l})&=&
\sum _{i=1}^{n+1} V({v_i})\cdot V(\mu_{k,l})\\
& =&
\frac {1}{a_{l}}
{(\sum_{i=1}^{n+1} a_i)}
\frac{\mult {(\mu_{k,l})}}
{\mult {(\sigma_{k})}},  
\end{eqnarray*} 
where $K_X=-\sum _{i=1}^{n+1}V(v_i)$ is 
a {\em{canonical divisor}} of $X$. 
For the procedure to compute intersection numbers, 
see \ref{int} or \cite[p.100]{fulton}. 
\end{say}

We note the following well-known fact. 

\begin{lem}\label{2727}
Let $X$ be an $n$-dimensional $\mathbb Q$-factorial complete normal variety with 
Picard number one. Assume that 
$X$ is toric. 
Then $X$ is an $n$-dimensional $\mathbb Q$-factorial toric Fano variety with Picard number one. 
\end{lem}

Let us recall the following easy lemma, which will play crucial roles in 
the proof of our main theorem:~Theorem \ref{main-thm}. 
The proof of Lemma \ref{kiso} is obvious by the description in \ref{tuitui}. 

\begin{lem}\label{kiso} 
We use the notations in {\em{\ref{tuitui}}}. 
We consider the sublattice 
$N'$ of $N$ spanned by 
$\{v_1, \cdots, v_{n+1}\}$. 
Then the natural inclusion 
$N'\to N$ induces a finite 
toric morphism $f:X'\to X$ from a 
weighted projective space $X'$ such 
that $f$ is {}\'etale in codimension one. 
In particular, $X(\Delta)$ is a weighted projective 
space if and only if $\{v_1, \cdots, v_{n+1}\}$ 
generates $N$. 
\end{lem}
For a toric description of {\em{weighted projective spaces}}, 
see \cite[Section 2]{fujino}. 

\begin{say}
In Lemma \ref{kiso}, we consider $C=V(\mu_{k, l})\simeq \mathbb P^1\subset X$ and 
the unique torus invariant curve $C'\subset X'$ such that 
$f(C')=C$. 
We set 
$$
m_{k, l}:=\deg (f|_{C'}: C'\to C) \in \mathbb Z_{>0}
$$ 
for every $(k, l)$. 
Then we can check that 
$$
m_{k, l}=\left|N(\mu_{k, l})/N'(\mu_{k, l})\right|
$$ 
by definitions, where 
$N'(\mu_{k, l})=N'/N'_{\mu_{k, l}}$ and 
$N(\mu_{k, l})=N/N_{\mu_{k, l}}$. 
Let $D$ be a Cartier divisor on $X$. 
Then we obtain 
$$
C\cdot D=\frac{1}{m_{k, l}}(C'\cdot f^*D)
$$ 
by the projection formula. 
Therefore, we have 
\begin{align*}
C\cdot V(v_k)&=V(\mu_{k, l})\cdot V(v_k)\\ 
&=\frac{\mult (\mu_{k, l})}{\mult (\sigma_l)}=\frac{\gcd (a_k, a_l)}
{m_{k, l}a_l}. 
\end{align*}
This is because 
$$
(C'\cdot f^*V(v_k))=\frac{\gcd(a_k, a_l)}{a_l}
$$ 
since $X'$ is a weighted projective space. 
\end{say}
\begin{say}[Lemma on the ACC]
We close this section with an easy lemma for the ascending 
chain condition. 
\begin{lem}\label{lem-ishi} 
We have the following elementary properties. 
\begin{enumerate}
\item If $A$ satisfies the ascending chain condition, then 
any subset $B$ of $A$ satisfies 
the ascending chain condition. 
\item If $A$ and $B$ satisfy the ascending chain condition, 
then so does 
$$
A+B=\{a+b\, |\, a\in A, b\in B\}. 
$$
\item If there exists a real number $t_0$ such that 
$$
A\subset \{x\in \mathbb R\, | \, x\geq t_0\}
$$ 
and $A\cap \{x\in \mathbb R\, |\, x>t\}$ is a finite set for any $t>t_0$, 
then $A$ satisfies the ascending chain condition. 
\end{enumerate}
\end{lem}
All the statements in Lemma \ref{lem-ishi} directly follow from 
definitions. 
\end{say}

\section{Proof of the main theorem}

In this section, we prove the main theorem of this paper:~Theorem \ref{main-thm}. 
We will freely use the notation in Section \ref{sec2}. 

\begin{proof}[Proof of {\em{Theorem \ref{main-thm}}}] 
Let $X$ be an $n$-dimensional $\mathbb Q$-factorial toric Fano variety with 
Picard number one as in \ref{tuitui}. 
It is sufficient to consider $\{v_1, \cdots, v_{n+1}\}$ with 
the condition 
$$
\frac{\mult(\mu_{1, 2})}{a_1\mult (\sigma_2)}\leq 
\frac{\mult(\mu_{k, l})}{a_k\mult (\sigma_l)}
$$
for every $(k, l)$. 
We note that 
$$
\frac{\mult(\mu_{k, l})}{a_k\mult (\sigma_l)}=
\frac{\mult(\mu_{k, l})}{a_l\mult (\sigma_k)}
$$
for every $k\ne l$. 
We also note that we can easily check that 
$$
l(X)=\underset{1\leq i\leq n+1}\min(-K_X\cdot V(v_i)) 
$$ 
(cf.~\cite[Proposition 14-1-2]{ma}). 
In our notation, we have 
$$
l(X)=\frac{\mult (\mu_{1, 2})}{a_1\mult (\sigma_2)}\sum _{i=1}^{n+1}
a_i 
$$ 
for this $\{v_1, \cdots, v_{n+1}\}$ by the formula in \ref{tuitui}. 
Therefore, we can write 
$$
\mathcal L_n^{\mathrm{toric}}=\left\{ \frac{\mult (\mu_{1, 2})}{a_1\mult 
(\sigma_2)}\sum _{i=1}^{n+1}a_i\left| 
\frac{\mult(\mu_{1, 2})}{a_1\mult (\sigma_2)}\leq 
\frac{\mult(\mu_{k, l})}{a_k\mult (\sigma_l)} 
\ {\text{for every $(k, l)$}}\right. 
\right\}. 
$$
It is sufficient to prove that 
$$
\mathcal M_i=\left\{\frac{\mult (\mu_{1, 2})}{a_1\mult (\sigma_2)}a_i\left|
\frac{\mult (\mu_{1, 2})}{a_1\mult (\sigma_2)}\leq 
\frac{\mult (\mu_{k, l})}{a_k\mult (\sigma_l)}
\ {\text{for every $(k, l)$}} \right.  \right\}
$$ satisfies the ascending chain condition. This is because 
$\mathcal L_n^{\mathrm{toric}}$ is contained 
in 
$$
\left\{\frac{\mult (\mu_{1, 2})}{\mult (\sigma_2)}\right\}+
\left\{\frac{\mult (\mu_{1, 2})}{\mult (\sigma_1)}\right\}
+\mathcal M_3+\cdots+\mathcal M_{n+1}. 
$$
We note that 
$$
\left\{\frac{\mult (\mu_{1, 2})}{\mult (\sigma_2)}\right\}, 
\left\{\frac{\mult (\mu_{1, 2})}{\mult (\sigma_1)}\right\} 
\subset \left\{\left. \frac{1}{m}\, \right| \, m\in \mathbb Z_{>0}\right\}. 
$$
Therefore, it is sufficient to prove the following proposition by Lemma \ref{lem-ishi}. 
\begin{prop}
For $3\leq i\leq n+1$, 
$\mathcal M_i\cap \{x\in \mathbb R\, | \, x>\varepsilon\}$ is 
a finite set for every $\varepsilon >0$. 
\end{prop}
From now on, we fix $i$ with $3\leq i\leq n+1$. 
Since 
$$
\mathcal M_i=\left\{\frac{\mult (\mu_{1, 2})}{a_1\mult (\sigma_2)}a_i\left|
\frac{\mult (\mu_{1, 2})}{a_1\mult (\sigma_2)}\leq 
\frac{\mult (\mu_{k, l})}{a_k\mult (\sigma_l)}
\ {\text{for every $(k, l)$}} \right.  \right\}, 
$$ 
we have 
\begin{align*}
\varepsilon &<\frac{\mult (\mu_{1, 2})}{a_1\mult (\sigma_2)} a_i
\\
&= \frac{\mult (\mu_{1, 2})}{a_1\mult (\sigma_2)}\cdot 
\frac{a_i\mult (\sigma_j)}{\mult (\mu_{i, j})}\cdot \frac{\mult (\mu_{i, j})}
{\mult (\sigma_j)}\\
&\leq 
\frac{\mult (\mu_{i, j})}{\mult (\sigma_j)} 
\end{align*}
for every $1\leq j\leq n+1$ with $j\ne i$. 
Therefore, we obtain 
$$
\frac{\mult (\sigma_j)}{\mult (\mu_{i, j})}\leq \llcorner \varepsilon ^{-1}\lrcorner
$$ for every $1\leq j\leq n+1$ with $j\ne i$, where $\llcorner \varepsilon^{-1}\lrcorner$ is 
the integer satisfying $\varepsilon^{-1}-1<\llcorner \varepsilon^{-1}\lrcorner\leq \varepsilon^{-1}$. 
We set 
$$
\mathbb Z(i, j)=\mathbb Z v_1+\cdots +
\mathbb Zv_{i-1}+\mathbb Zv_{i+1}+\cdots 
+\mathbb Zv_{j-1}+\mathbb Zv_{j+1}+\cdots +
\mathbb Zv_{n+1}
$$ 
for $j\ne i$ and 
$$\mathbb Z(j)=\mathbb Zv_1+\cdots +\mathbb Zv_{j-1}+
\mathbb Zv_{j+1}+\cdots +\mathbb Zv_{n+1}. 
$$ 
We consider the following diagram. 
$$
\xymatrix{
&0 \ar[d]& 0\ar[d]& 0\ar[d]& \\ 
0\ar[r]& \mathbb Z(i, j)\ar[d]\ar[r]& \mathbb Z(j)\ar[d]\ar[r]
& \mathbb Z\ar[d]\ar[r]& 0\\
0\ar[r] & N_{\mu_{i, j}}\ar[d]\ar[r]&
N\ar[d]^{\pi_j}\ar[r]&N/N_{\mu_{i, j}}\ar[d]
\ar[r] &0\\
0\ar[r] & N_{\mu_{i, j}}/\mathbb Z(i, j)\ar[d]
\ar[r]&N/\mathbb Z(j)\ar[r]^{p_j}\ar[d]& \mathcal A^j _{(i, j)}\ar[r] \ar[d]& 0\\ 
& 0& 0& 0&  
}
$$
We note that 
$$
\left|\mathcal A^j_{(i, j)}\right|=\frac{\mult (\sigma_j)}{\mult (\mu_{i, j})}\leq \llcorner 
\varepsilon ^{-1}\lrcorner. 
$$
Therefore, for any $v\in N$, we have 
$$
p_j\circ \pi_j\left((\llcorner \varepsilon ^{-1}\lrcorner)!v\right)=0
$$
in $\mathcal A^j_{(i, j)}$. 
Thus, 
$$
\pi_j\left((\llcorner \varepsilon^{-1}\lrcorner)!v\right)\in N_{\mu_{i, j}}/\mathbb Z(i, j). 
$$
This holds for every $1\leq j\leq n+1$ with $j\ne i$. 
Let us consider the natural projection $\pi:N\to N/N'$ where 
$N'=\sum _{k=1}^{n+1}\mathbb Zv_k$. Then, by the above argument, 
we obtain that 
$$\pi\left((\llcorner \varepsilon ^{-1}\lrcorner)!v\right)\in 
\bigcap _{j\ne i}N_{\mu_{i, j}}/(N'\cap N_{\mu_{i, j}})\subset N/N'. $$ 
\begin{claim}$\pi\left((\llcorner \varepsilon ^{-1}\lrcorner)!v\right)=0$ in $N/N'$, equivalently, 
$(\llcorner \varepsilon ^{-1}\lrcorner)!v\in N'$. 
\end{claim}
\begin{proof}[Proof of Claim] 
By replacing $v_i$ with $v_{n+1}$, we may assume that $i=n+1$. 
We embed $N$ and $N'$ into $\mathbb Q^n$ by setting 
$v_1=(1, 0, \cdots, 0)$, $v_2=(0, 1, 0, \cdots, 0)$, $\cdots$, and $v_n=(0, \cdots, 0, 1)$. 
Then it is easy to see that 
$$
\bigcap_{1\leq j\leq n}(N'+N_{\mu_{n+1, j}})=N'. 
$$
On the other hand, we have 
$$
N_{\mu_{n+1, j}}/(N'\cap N_{\mu_{n+1, j}})\simeq (N'+N_{\mu_{n+1, j}})/N'
$$ 
for $1\leq j\leq n$. 
Therefore, 
$$\pi\left((\llcorner \varepsilon ^{-1}\lrcorner)!v\right)\in 
\bigcap _{1\leq j\leq n}N_{\mu_{n+1,j}}/(N'\cap N_{\mu_{n+1, j}})\subset N/N'$$ 
implies that $$\pi\left((\llcorner \varepsilon^{-1}\lrcorner)!v\right)=0$$ in $N/N'$, equivalently, 
$$
(\llcorner \varepsilon^{-1}\lrcorner)!v\in N'.
$$
This completes the proof of Claim. 
\end{proof}
Thus, we obtain $$1\leq m_{1, 2}\leq (\llcorner \varepsilon ^{-1}\lrcorner)!.$$ 
Moreover, 
$$
\varepsilon <\frac{\mult (\mu_{1, i})}{\mult (\sigma_1)}=\frac{\gcd(a_1, a_i)}{m_{1, i}a_1}\leq 
\frac{\gcd(a_1, a_i)}{a_1}. 
$$ 
By the same way, we obtain 
$$
\varepsilon <\frac{\gcd(a_2, a_i)}{a_2}. 
$$ 
We note the following obvious inequality
$$
\frac{\mult (\mu_{1, 2})}{a_1\mult (\sigma_2)}
a_i\leq \frac{\mult (\mu_{1, 2})}{a_1\mult (\sigma_2)}
\cdot \frac{a_i\mult (\sigma_2)}{\mult (\mu_{2, i})}\leq 1. 
$$
Since $a_1$, $a_2$, and $a_i$ are positive integers, we have 
$$
\gcd(l, a_i)=\frac{\gcd(a_1, a_i)\cdot \gcd(a_2, a_i)}{\gcd(d, a_i)}
$$
where $d:=\gcd(a_1, a_2)$ and $l:=\lcm (a_1, a_2)=\frac{a_1a_2}{d}$.  
Therefore, we obtain 
$$
\frac{\gcd(l, a_i)}{l}=\frac{\gcd(a_1, a_i)}{a_1}\cdot \frac{\gcd(a_2, a_i)}{a_2}\cdot 
\frac{d}{\gcd(d, a_i)}>\varepsilon ^2\frac{d}{\gcd(d, a_i)}\geq \varepsilon ^2. 
$$
This means that 
$$
\frac{l}{\gcd(l, a_i)}\leq \varepsilon^{-2}. 
$$
Thus, we have 
\begin{align*}
1&\geq \frac{\mult (\mu_{1, 2})}{a_1\mult (\sigma_2)}a_i=
\frac{a_i}{m_{1, 2}l}=\frac{\gcd(l ,a_i)}{l}\cdot \frac{a_i}{m_{1, 2}\gcd(l, a_i)}\\
&\geq \varepsilon^2\frac{a_i}{m_{1, 2}\gcd(l, a_i)}. 
\end{align*}
So, we obtain 
$$
\frac{a_i}{\gcd(l, a_i)}\leq \varepsilon ^{-2}m_{1, 2}\leq \varepsilon^{-2}(\llcorner \varepsilon^{-1}\lrcorner)!. 
$$ 
On the other hand, 
$$
\frac{\mult(\mu_{1, 2})}{a_1\mult (\sigma_2)}a_i=\frac{a_i}{m_{1, 2}l}.  
$$ 
We note that 
$$
\frac{a_i}{m_{1, 2}l}=\frac{\displaystyle{\frac{a_i}{\gcd(l, a_i)}}}
{\displaystyle{m_{1, 2}\frac{l}{{\gcd(l, a_i)}}}}. 
$$
This implies that $\mathcal M_i\cap \{x\in \mathbb R\, |\, x>\varepsilon\}$ 
is a finite set. This is because 
$$
\frac{a_i}{\gcd(l, a_i)}, \ \frac{l}{\gcd(l, a_i)}, \ m_{1, 2}
$$ are 
positive integers and 
$$
\frac{a_i}{\gcd(l, a_i)}\leq \varepsilon ^{-2}(\llcorner \varepsilon ^{-1}\lrcorner)!,\ 
\frac{l}{\gcd(l, a_i)}\leq\varepsilon ^{-2}, \ m_{1, 2}\leq (\llcorner 
\varepsilon ^{-1}\lrcorner)!. 
$$ 
Thus we proved the proposition and 
$\mathcal L_n^{\mathrm{toric}}$ satisfies the ascending chain condition. 
\end{proof}

\end{document}